\documentclass[11pt]{article}
\usepackage[english]{babel}
\usepackage{amssymb,amsmath,amsthm}
\newtheorem{theorem}{Theorem}[section]
\newtheorem{lemma}[theorem]{Lemma}
\newtheorem{proposition}[theorem]{Proposition}
\newtheorem{corollary}[theorem]{Corollary}

{\theoremstyle{definition}}
{\theoremstyle{definition}}
{\theoremstyle{definition}}
{\theoremstyle{definition}}

\newtheorem*{q1}{Question F1}
\newtheorem*{q2}{Question F2}
\newtheorem*{q3}{Question F3}
\newtheorem*{q4}{Question F4}

\numberwithin{equation}{section}

\def\C{{\mathbb C}}
\def\A{{\mathbb A}}
\def\N{{\mathbb N}}
\def\Z{{\mathbb Z}}
\def\R{{\mathbb R}}
\def\Q{{\mathbb Q}}
\def\K{{\mathbb K}}
\def\T{{\bf T}}
\def\F{{\cal F}}
\def\uu{{\cal U}}
\def\epsilon{\varepsilon}
\def\kappa{\varkappa}
\def\phi{\varphi}
\def\leq{\leqslant}
\def\geq{\geqslant}

\def\dim{\hbox{\tt dim}\,}
\def\ker{\hbox{\tt ker}\,}

\def\ssub#1#2{#1_{{}_{{\scriptstyle #2}}}}

\def\uu{{\mathcal U}}

\def\dimr{\ssub{\hbox{\tt dim}}{\R}}
\def\expa{\ssub{\hbox{\tt exp}}{\A}}

\title{Hypercyclic tuples of operators on $\C^n$ and $\R^n$}

\author{Stanislav Shkarin}

\date{}

\begin{document}

\maketitle

\begin{abstract} A tuple $(T_1,\dots,T_n)$ of continuous linear
operators on a topological vector space $X$ is called hypercyclic if
there is $x\in X$ such that the the orbit of $x$ under the action of
the semigroup generated by $T_1,\dots,T_n$ is dense in $X$. This
concept was introduced by N.~Feldman, who have raised 7 questions on
hypercyclic tuples. We answer  those 4 of them, which can be dealt
with on the level of operators on finite dimensional spaces. In
particular, we prove that the minimal cardinality of a hypercyclci
tuple of operators on $\C^n$ (respectively, on $\R^n$) is $n+1$
(respectively, $\frac n2+\frac{5+(-1)^n}{4}$), that there are
non-diagonalizable tuples of operators on $\R^2$ which possess an
orbit being neither dense nor nowhere dense and construct a
hypercyclic 6-tuple of operators on $\C^3$ such that every operator
commuting with each member of the tuple is non-cyclic.
\end{abstract}



\section{Introduction \label{s1}}\rm

Throughout the article $\K$ stands for either the field $\C$ of
complex numbers or the field $\R$ of real numbers, $\Q$ is the field
of rational numbers, $\Z$ is the set of integers and $\Z_+$ is the
set of non-negative integers. Symbol $L(X)$ stands for the space of
continuous linear operators on a topological vector space $X$. A
family ${\cal F}=\{F_a:a\in A\}$ of continuous maps from a
topological space $X$ to a topological space $Y$ is called {\it
universal} if there is $x\in X$ for which the orbit $O({\cal
F},x)=\{F_ax:a\in A\}$ is dense in $Y$. Such an $x$ is called a {\it
universal element} for $\cal F$. We use the symbol $\uu({\cal F})$
to denote the set of universal elements for $\cal F$. If $X$ is a
topological vector space, $n\in\N$ and $\T=(T_1,\dots,T_n)\in
L(X)^n$, then we call $\T$ a {\it commuting $n$-tuple} if
$T_jT_k=T_kT_j$ for $1\leq j,k\leq n$. An $n$-tuple
$\T=(T_1,\dots,T_n)\in L(X)^n$ is called {\it hypercyclic} if it is
commuting and the semigroup
$$
\F_{\T}=\{T_1^{k_1}\dots T_n^{k_n}:k_j\in\Z_+\}
$$
is universal. The elements of $\uu(\F_\T)$ are called {\it
hypercyclic vectors for} $\T$. This concept was introduced and
studied by Feldman \cite{feld}. In the case $n=1$, it becomes the
conventional hypercyclicity, which has been widely studied, see the
book \cite{bama-book} and references therein. It turns out that,
unlike for the classical hypercyclicity, there are hypercyclic
tuples of operators on finite dimensional spaces. Namely, in
\cite{feld} it is shown that $\C^n$ admits a hypercyclic
$(n+1)$-tuple of operators. It is also shown that for every tuple of
operators on $\C^n$, any orbit is either dense or is nowhere dense
and that the latter fails for tuples of operators on $\R^n$. The
article \cite{feld} culminates in raising 7 open questions. We
answer those 4 of them (Questions~2, 3, 5 and 6 in the list) that
can be dealt with on the level of operators on finite dimensional
spaces.

\begin{q1} If $\T$ is a hypercyclic tuple, then must $\F_\T$ contain
a cyclic operator$\,?$
\end{q1}

\begin{q2} If $\T$ is a hypercyclic tuple, is there a cyclic operator
that commutes with $\T\,?$
\end{q2}

\begin{q3} Is there a hypercyclic $n$-tuple of operators on $\C^n\,?$
\end{q3}

\begin{q4} Are there non-diagonalizable commuting tuples of operators on
$\R^n$, possessing orbits that are neither dense nor nowhere
dense$\,?$
\end{q4}

We answer the above questions by considering dense additive
subsemigroups of $\R^n$ and by studying cyclic commutative
subalgebras of full matrix algebras.

As usual, we identify $L(\K^n)$ with the $\K$-algebra of $n\times n$
matrices with entries from $\K$. In this way $L(\R^n)$ will be
treated as an $\R$-subalgebra of the $\C$-algebra $L(\C^n)$. For a
commutative subalgebra $\A$ of $L(\K^n)$, by $\A^+$ we denote the
subalgebra of $S\in L(\K^n)$ commuting with each element of $\A$.
Clearly $\A$ is a subalgebra of $\A^+$, however $\A^+$ may fail to
be commutative. Recall that a subalgebra $\A$ of $L(\K^n)$ is called
{\it cyclic} if there is $x\in\K^n$ such that $\{Ax:A\in\A\}=\K^n$.
For a commuting tuple $\T$ of operators on $\K^n$ we denote the
unital subalgebra of $L(\K^n)$ generated by $\T$ by the symbol
$\A_{\T}$. Obviously, $\A_\T$ is commutative and therefore
$\A_\T\subseteq \A_\T^+$ and $\A_{\T}$ is cyclic if $\T$ is
hypercyclic. Thus each hypercyclic tuple of operators on $\K^n$ lies
in a unital commutative cyclic subalgebra of $L(\K^n)$.

Recall that a {\it character} on a complex algebra $\A$ is a
non-zero algebra homomorphism from $\A$ to $\C$. It is well-known
that the set of characters is always linearly independent in the
space of linear functionals on $\A$. Moreover, the set of characters
on a unital commutative Banach algebra is always non-empty. Hence an
$n$-dimensional unital commutative complex algebra has at least $1$
and at most $n$ characters. For a finite dimensional complex algebra
$\A$, $\kappa(\A)$ stands for the number of characters on $\A$. Thus
\begin{equation}\label{kappa}
\text{$1\leq \kappa(\A)\leq n$ \ for every $n$-dimensional unital
commutative complex algebra $\A$.}
\end{equation}
For a real algebra $\A$, a {\it character} on $\A$ is a non-zero
$\R$-algebra homomorphism from $\A$ to $\C$. It is well-known and
easy to see that each character on $\A$ has a unique extension to a
character on the complex algebra $\A_\C=\A\oplus i\A=\C\otimes\A$,
being the complexification of $\A$. Thus the characters on $\A$ and
the characters on $\A_\C$ are in the natural one-to-one
correspondence. If $\chi$ is a character on a real algebra $\A$,
then its complex conjugate $\overline{\chi}$ is also a character on
$\A$. Clearly, $\chi(\A)$ being a non-trivial $\R$-subalgebra of
$\C$ must coincide either with $\R$ or with $\C$. The characters of
the second type come in complex conjugate pairs. For a finite
dimensional $\R$-algebra, by $\kappa_0(\A)$ we denote the number of
complex conjugate pairs of characters $\chi$ on $\A$ satisfying
$\chi(\A)=\C$. Similarly, $\kappa_1(\A)$ stands for the number of
real-valued characters on $\A$. Since
$\kappa(\A_\C)=\kappa_1(\A)+2\kappa_0(\A)$, (\ref{kappa}) implies
that
\begin{equation}\label{kappa0}
\text{$1\leq 2\kappa_0(\A)+\kappa_1(\A)\leq n$ \ for every unital
commutative $n$-dimensional real algebra $\A$.}
\end{equation}

\begin{lemma}\label{cyc} Let $\A$ be a commutative cyclic subalgebra
of $L(\K^n)$. Then $I\in\A$, $\A=\A^+$ and $\dim\A=n$. In
particular, $1\leq\kappa(\A)\leq n$ if $\K=\C$ and $1\leq
\kappa_1(\A)+2\kappa_0(\A)\leq n$ if $\K=\R$.
\end{lemma}

\begin{proof} Let $x\in\K^n$ be a cyclic vector for $\A$. Then the
linear map $A\mapsto Ax$ from $\A$ to $\K^n$ is surjective and
therefore $\dim\A\geq n$. Now let $B\in\A^+$ be such that $Bx=0$.
Since $B$ commutes with each member of $\A$, $\ker B$ is invariant
for every element of $\A$ and therefore
$\K^n=\{Ax:A\in\A\}\subseteq\ker B$. That is, $B=0$ whenever
$B\in\A^+$ and $Bx=0$. Thus the linear map $B\mapsto Bx$ from $\A^+$
to $\K^n$ is injective and therefore $\dim\A^+\leq n$. Since also
$\A\subseteq\A^+$, we get $\A=\A^+$ and $\dim\A=n$. Since
$I\in\A^+$, we have $I\in\A$. The required estimates now follow from
(\ref{kappa}) and (\ref{kappa0}).
\end{proof}

The significance of Lemma~\ref{cyc} becomes clear in view of the
next two results.

\begin{theorem}\label{cyc2} Let $\A$ be a commutative cyclic
subalgebra of $L(\C^n)$ and $m=2n-\kappa(\A)+1$. Then $\A$ contains
a hypercyclic $m$-tuple of operators on $\C^n$ and $\A$ contains no
hypercyclic $(m-1)$-tuples.
\end{theorem}

\begin{theorem}\label{cyc22} Let $\A$ be a commutative cyclic
subalgebra of $L(\R^n)$ and $r=n-\kappa_0(\A)+1$. Then $\A$ contains
a hypercyclic $r$-tuple of operators on $\R^n$, while every orbit of
every $(r-1)$-tuple of operators from $\A$ is nowhere dense.
\end{theorem}

The subalgebra $\A_{\rm diag}$ of $L(\C^n)$ consisting of all
operators with diagonal matrices is commutative, cyclic and has
exactly $n$ characters. It is also easy to show that a commutative
subalgebra $\A$ of $L(\C^n)$ has $n$ characters if and only if $\A$
is conjugate to $\A_{\rm diag}$. Now if $\A_0$ is the subalgebra of
$L(\C^2)$
consisting of the matrices of the form $\begin{pmatrix}a&b\\
0&a\end{pmatrix}$ with $a,b\in\C$, then then $\A_0$ is commutative
and cyclic and $\kappa(\A_0)=1$. The direct sum of $\A_0$ and $n-2$
copies of $L(\C^1)$ provides a commutative cyclic subalgebra of
$L(\C^n)$ with exactly $n-1$ characters. By Theorem~\ref{cyc2}, it
contains a hypercyclic $(n+2)$-tuple. Thus Theorem~\ref{cyc2}
implies the following corollary, which answers Question~F3
negatively.

\begin{corollary}\label{cn} The minimal cardinality of a
hypercyclic tuple of operators on $\C^n$ is exactly $n+1$. The
minimal cardinality of a non-diagonalizabe hypercyclic tuple of
operators on $\C^n$ is $n+2$.
\end{corollary}

If we consider the subalgebra $\A_1$ of $L(\R^2)$ consisting of
matrices of the form $\begin{pmatrix}a&b\\ -b&a\end{pmatrix}$ with
$a,b\in\R$, then $\A_1$ is commutative and cyclic and
$\kappa_0(\A_1)=1$. Then the direct sum $\A_m$ of $m$ copies of
$\A_1$ is a cyclic commutative subalgebra of $L(\R^{2m})$ with
$\kappa_0(\A_m)=m$. Similarly, the direct sum $\A'_m$ of $L(\R^1)$
and $m$ copies of $\A_1$ is a cyclic commutative subalgebra of
$L(\R^{2m+1})$ with $\kappa_0(\A'_m)=m$. Thus for every positive
integer $n$, there is a commutative cyclic subalgebra $\A$ of
$L(\R^n)$ with $\kappa_0(\A)$ being exactly the integer part of
$\frac n2$. This observation together with Theorem~\ref{cyc22}
implies the following corollary.

\begin{corollary}\label{rn} Let $n$ be a positive integer and
$k=\frac n2+\frac{5+(-1)^{n+1}}4$. Then there is a hypercyclic
$k$-tuple of operators on $\R^n$. Moreover, every orbit of every
commuting  $(k-1)$-tuple of operators on $\R^n$ is nowhere dense.
\end{corollary}

Note that $\frac n2+\frac{5+(-1)^{n+1}}4$ is $\frac n2+1$ if $n$ is
even and is $\frac{n+3}2$ if $n$ is odd. The following result
provides negative answers to Questions~F1 and~F2.

\begin{proposition}\label{cy} There is a hypercyclic
$6$-tuple $\T$ of operators on $\C^3$ such that every operator
commuting with $($all members of$\,)$ $\T$ is non-cyclic. Moreover,
there is a hypercyclic $4$-tuple $\T$ of operators on $\R^3$ such
that every operator commuting with $\T$ is non-cyclic.
\end{proposition}

Question~F4 admits an affirmative answer by means of the following
proposition.

\begin{proposition}\label{F4} There exist positive numbers
$a_1,a_2,a_3$ and non-zero real numbers $b_1,b_2,b_3$ such that the
commuting triple $\T=(T_1,T_2,T_3)$ of non-diagonalizable $($even
when considered as members of $L(\C^2))$ operators
$T_j=\begin{pmatrix}a_j&b_j\\
0&a_j\end{pmatrix}$ on $\R^2$ has the following properties$:$
\begin{itemize}\itemsep-2pt
\item[\rm (a1)]$\T$ is non-hypercyclic$;$
\item[\rm (a2)]the orbit of $x=(x_1,x_2)\in\R^2$ with respect to
$\T$ is contained and is dense in the half-plane
$\Pi=\{(s,t)\in\R^2:t>0\}$ if $x_2>0$.
\end{itemize}
\end{proposition}

\section{Dense additive subsemigroups of $\R^n$}

We start with the following trivial lemma, whose proof we give for
sake of completeness.

\begin{lemma}\label{subs1} Let $H$ be an additive subgroup of $\R^n$
with at most $n$ generators. Then $H$ is nowhere dense in $\R^n$.
\end{lemma}

\begin{proof} Let $k\leq n$ and $x_1,\dots,x_k$ be generators of
$H$. If the linear span $L$ of $x_1,\dots,x_k$ differs from $\R^n$,
then $H$ is nowhere dense as a subset of the closed nowhere dense
set $L$. It remains to consider the case $L=\R^n$. Since $k\leq n$,
it follows that $k=n$ and the vectors $x_1,\dots,x_k$ form a basis
in $\R^n$. Hence $H$ is a discrete lattice in $\R^n$ and therefore
is nowhere dense in $\R^n$.
\end{proof}

\begin{lemma}\label{subs21} Let $\alpha_1,\dots,\alpha_n$ be $n$
positive numbers linearly independent over $\Q$. Then
$$
\Omega=\{(m_1-m_0\alpha_1,\dots,m_n-m_0\alpha_n):m_0,\dots,m_n\in\Z_+\}\
\ \ \text{is dense in $\R^n$.}
$$
\end{lemma}

\begin{proof} Let $g=\alpha+\Z^n\in{\mathbb T}^n=\R^n/\Z^n$. Since the components of
$\alpha=(\alpha_1,\dots,\alpha_n)$ are linearly independent over
$\Q$, the classical Kronecker theorem implies that $\{mg:m\in\Z_+\}$
is dense in the compact metrizable topological group ${\mathbb
T}^n$. It follows that for every $x\in\R^n$, we can find a strictly
increasing sequence $\{m_{0,j}\}_{j\in\Z_+}$ in $\Z_+$ such that
$m_{0,j}g\to -x+\Z^n$. Hence we can pick sequences
$\{m_{l,j}\}_{j\in\Z_+}$ of integers for $1\leq l\leq n$ such that
$m_{l,j}-m_{0,j}\alpha_l\to x_l$ as $j\to\infty$ for $1\leq l\leq
n$. Since $m_{0,j}\to +\infty$ and $\alpha_l>0$, we have
$m_{l,j}\geq 0$ for all sufficiently large $j$. It follows that $x$
is an accumulation point of $\Omega$. Since $x\in\R^n$ is arbitrary,
$\Omega$ is dense in $\R^n$.
\end{proof}

\begin{lemma}\label{subs3} Let $x_1,\dots,x_n$ be a basis in $\R^n$
and $G$ be a finite abelian group $($carrying the discrete
topology$)$ with an $n+1$-element generating set
$\{g_0,\dots,g_{n}\}$. Then there is $x_0\in\R^n$ such that the
subsemigroup of $G\times\R^n$ generated by
$(g_0,x_0),\dots,(g_n,x_n)$ is dense in $G\times\R^n$.
\end{lemma}

\begin{proof} For the sake of homogeneity, we use the additive
notation for the operation on $G$. Let $\alpha_1,\dots,\alpha_n$ be
positive numbers linearly independent over $\Q$ and
$x_0=-(\alpha_1x_1+{\dots}+\alpha_nx_n)$. It suffices to show that
the subsemigroup $H$ of $G\times\R^n$ generated by
$(g_0,x_0),\dots,(g_n,x_n)$ is dense in $G\times\R^n$. Let $m$ be
the order of $G$ and $h\in G$. Pick positive integers
$j_0,\dots,j_n$ such that $h=j_0g_0+{\dots}+j_ng_n$. Then
$$
\text{$H_h\subset H$, where}\ \ \
H_h=\{(j_0+k_0m)(g_0,x_0)+{\dots}+(j_n+k_nm)(g_n,x_n):k_j\in\Z_+\}.
$$
Denote $y=j_0x_0+{\dots}+j_nx_n$. Using the equalities
$h=j_0g_0+{\dots}+j_ng_n$, $mg_l=0$ for $0\leq l\leq n$ and the
above display, we obtain
$$
H_h=\{(h,y+m((k_1-\alpha_1k_0)x_1+{\dots}+(k_n-\alpha_nk_0)x_n)):k_j\in\Z_+\}.
$$
By Lemma~\ref{subs21}, $H_h$ is dense in $\{h\}\times \R^n$. Since
$h$ is an arbitrary element of $G$ and $H_h\subset H$, $H$ is dense
in $G\times \R^n$.
\end{proof}

Applying Lemma~\ref{subs3} in the case $G=\{0\}$ and
$g_0={\dots}=g_n=0$, we get the following corollary.

\begin{corollary}\label{subs2} Let $x_1,\dots,x_n$ be a basis in $\R^n$.
Then there exists $x_0\in\R^n$ such that the additive subsemigroup
$H$ in $\R^n$ generated by $\{x_0,\dots,x_n\}$ is dense in $\R^n$.
\end{corollary}

\section{Properties of the exponential map}

Throughout this section
\begin{align*}
&\text{$\A$ is a finite dimensional unital commutative algebra over
$\K$}
\\
&\qquad\text{and $e^a=\sum\limits_{n=0}^\infty \frac{a^n}{n!}\in \A$
for $a\in\A$.}
\end{align*}
The symbol $\expa$ stands for the map $a\mapsto e^a$ from $\A$ to
itself. Since $\A$ is commutative, $e^{a+b}=e^ae^b$ and $e^a$ is
invertible with the inverse $e^{-a}$ for any $a,b\in \A$. Thus
$\expa$ is a homomorphism from the additive group $(\A,+)$ to the
(abelian) group $(\A^\star,\cdot)$ of invertible elements of $\A$.
Since $\expa$ is a smooth map with the Jacobian having rank $\dim\A$
at every point, $\expa$ is a local diffeomorphism and therefore a
local homeomorphism from $\A$ to $\A$. The following lemma
represents a well-known fact. We sketch its proof for the sake of
convenience.

\begin{lemma}\label{cyc1} Let $\K=\C$ and $\Omega$ be the set of all
characters on $\A$. Then there are $p_\chi\in\A$ for $\chi\in\Omega$
such that
\begin{itemize}\itemsep-2pt
\item[\rm(a)]$p_\chi p_\phi=\delta_{\chi,\phi}p_\chi$ for every
$\chi,\phi\in\Omega$ and $\sum\limits_{\chi\in\Omega}p_\chi={\bf
1};$
\item[\rm(b)] the spectrum of $ap_\chi$ in the subalgebra $\A_\chi=p_\chi
\A$ is $\{\chi(a)\}$ for every $a\in\A$ and $\chi\in\Omega$.
\end{itemize}
Moreover, conditions {\rm (a)} and {\rm (b)} determine the
idempotents $p_\chi$ uniquely.
\end{lemma}

{\it Sketch of the proof.} \ By (\ref{kappa}), $\Omega$ is a finite
set. Since $\Omega$ is linearly independent, we can pick $a\in\A$
such that the numbers $\chi(a)$ for $\chi\in \Omega$ are pairwise
different. For $\chi\in\Omega$ we take a circle $\Gamma_\chi$ on the
complex plane centered at $\chi(a)$ and such that $\phi(a)$ for
$\phi\neq \chi$ are all outside the closed disk encircled by
$\Gamma_\chi$. Since $\sigma(a)=\{\phi(a):\phi\in\Omega\}$,
$\Gamma_\chi\cap \sigma(a)=\varnothing$ and we can consider
$$
p_\chi=\frac1{2\pi i}\oint_{\Gamma_\chi}(a-\zeta{\bf
1})^{-1}\,d\zeta\in \A,
$$
where $\Gamma_\chi$ is encircled counterclockwise. In exactly the
same way as when dealing with the Riesz projections, it is
straightforward to see that $p_\chi$ are pairwise orthogonal
idempotents and form a partition of the identity. Thus (a) is
satisfied. It is also a routine exercise to check that
$b-\sum\limits_{\chi\in\Omega}\chi(b)p_\chi$ is nilpotent for every
$b\in\A$, which implies (b). The uniqueness is also standard: one
has to check that any $p_\chi$ satisfying (a) and (b) must actually
satisfy the above display as well. \qed

\begin{lemma}\label{expc} Let $\K=\C$ and $k=\kappa(\A)$. Then the
homomorphism $\expa:\A\to\A^\star$ is a surjective local
homeomorphism, whose kernel \ $\ker \expa$ is a subgroup of $(\A,+)$
generated by $k$ linearly independent elements.
\end{lemma}

\begin{proof} We already know that $\expa$ is a local homeomorphism.
Let us verify that $\expa:\A\to\A^\star$ is surjective. Indeed, let
$a\in\A^\star$. Since $0\notin\sigma(a)$, we can pick
$w\in\C\setminus\{0\}$ such that the ray $L=\{tw:t\geq 0\}$ does not
meet $\sigma(a)$. Now let $\phi:\C\setminus L\to\C$ be a branch of
the logarithm function. Since $\phi$ is holomorphic on the open set
$\C\setminus L$ containing the spectrum of $a$ we can define
$b=\phi(a)$ in the standard holomorphic functional calculus sense.
Since $e^{\phi(z)}=z$ for every $z\in\C\setminus L$, it follows that
$e^b=a$. Hence $\expa:\A\to\A^\star$ is surjective.

Now let $\Omega$ be the set of characters on $\A$ and $p_\chi$ for
$\chi\in\Omega$ be the idempotents provided by Lemma~\ref{cyc1}. It
is well-known and easy to see that the exponential of a linear map
$S$ on $\C^n$ is the identity if and only if $S$ is diagonalizable
and $\sigma(S)\subseteq 2\pi i\Z$. Thus $a\in\A$ belongs to $K=\ker
\expa$ if and only if the linear map $x\mapsto ax$ on $\A$ is
diagonalizable and its spectrum is contained in $2\pi i\Z$. From the
conditions (a) and (b) of Lemma~\ref{cyc1} it follows that $a\in K$
if and only if $a$ is a linear combination of $p_\chi$ with
coefficients from $2\pi i\Z$. Thus the $k$ linearly independent
elements $2\pi i p_\chi$ for $\chi\in\Omega$ generate $K$ as a
subgroup of $(\A,+)$.
\end{proof}

The case $\K=\R$ turns out to be slightly more sophisticated. We
start with the following curious elementary lemma.

\begin{lemma}\label{expsq} Let $\K=\R$. Then for $a\in\A$ the
following statements are equivalent$:$
\begin{itemize}\itemsep-2pt
\item[\rm(i)]there is $b\in\A$ such that $a=e^b;$
\item[\rm(ii)]$a$ is invertible and there is $c\in\A$ such that $a=c^2.$
\end{itemize}
\end{lemma}

\begin{proof} If $a=e^b$, then $a$ is invertible and
$a=(e^{b/2})^2$. Thus (i) implies (ii). Assume now that $a$ is
invertible and $a=c^2$ for some $c\in\A$. Then $c$ is also
invertible in $\A$ and therefore $c$ is invertible in the
complexification $\A_\C=\A\oplus i\A$. By Lemma~\ref{expc}, there is
$d\in\A_\C$ such that $e^d=c$. Consider the involution
$(x+iy)^\dagger=x-iy$ on $\A_\C$, where $x,y\in\A$. Then
$e^{d^\dagger}=(e^d)^\dagger=c^\dagger=c$. Thus
$e^{d+d^\dagger}=e^de^{d^\dagger}=c^2=a$. It remains to notice that
$d+d^\dagger\in\A$ to conclude that (ii) implies (i).
\end{proof}

\begin{lemma}\label{expr} Let $\K=\R$, $k=\kappa_0(\A)$ and
$m=\kappa_1(\A)$. Then $\expa:\A\to\A^\star$ is a local
homeomorphism and there is a finite subgroup $G$ of $\A^\star$
isomorphic to $\Z_2^{m}$ such that $\A^\star$ is the topological
internal direct product of its subgroups $\expa(\A)$ and $G$.
Furthermore, $\ker \expa$ is a discrete subgroup of $(\A,+)$
generated by $k$ linearly independent elements.
\end{lemma}

\begin{proof}Since $\kappa(\A_\C)=2\kappa_0(\A)+\kappa_1(\A)=2k+m$,
we can enumerate the $(2k+m)$-element set $\Omega$  of characters on
$\A_\C$ in the following way: $\Omega=\{\chi_1,\dots,\chi_{2k+m}\}$,
where $\chi_{2k+j}(\A)=\R$ for $1\leq j\leq m$ and
$\chi_{m+j}(a)=\overline{\chi_j(a)}$ for $1\leq j\leq k$ for any
$a\in\A$. Consider the involution $a\mapsto a^\dagger$ on
$\A_\C=\A\oplus i\A$ defined by the formula $(b+ic)^\dagger=b-ic$
for $b,c\in\A$. The above properties of $\chi_j$ can be rewritten in
the following way:
\begin{equation}\label{bla}
\chi_{2k+j}(a^\dagger)=\overline{\chi_{2k+j}(a)}\ \ \text{and}\ \
\chi_{k+l}(a^\dagger)=\overline{\chi_l(a)}\ \ \text{for $a\in\A_\C$,
$1\leq j\leq m$ and $1\leq l\leq k$.}
\end{equation}
Let $p_j=p_{\chi_j}\in\A_\C$ for $1\leq j\leq 2k+m$ be the
idempotents provided by Lemma~\ref{cyc1} applied to the complex
algebra $\A_\C$. By (\ref{bla}), the idempotents
$p_{k+1}^\dagger,\dots,p_{2k}^\dagger,p_1^\dagger,\dots,p_k^\dagger,
p_{2k+1}^\dagger,\dots,p_{2k+m}^\dagger$ satisfy the conditions (a)
and (b) of Lemma~\ref{cyc1} in relation to the characters
$\chi_1,\dots,\chi_{2k+m}$. By the uniqueness part of
Lemma~\ref{cyc1}, $p_j=p_j^\dagger$ for $2k+1\leq j\leq 2k+m$ and
$p_j^\dagger=p_{k+j}$ for $1\leq j\leq k$. It immediately follows
that
\begin{equation}\label{cla1}
\text{$p_j\in\A$ for $2k+1\leq j\leq 2k+m$ and
$zp_j+\overline{z}p_{k+j}\in\A$ for $1\leq j\leq k$ and any
$z\in\C$.}
\end{equation}

As we have already observed in the proof of Lemma~\ref{expc}, the
kernel $K_0$ of the homomorphism $a\mapsto e^a$ from $\A_\C$ to
$\A_\C^\star$ is given by
\begin{equation*}
K_0=\{2\pi i(n_1p_1+{\dots}+n_{2k+m}p_{2k+m}):n_j\in\Z\}.
\end{equation*}
Clearly the kernel $K$ of $\expa$ satisfies $K=K_0\cap\A$. Let
$b=2\pi i(n_1p_1+{\dots}+n_{2k+m}p_{2k+m})$ with $n_j\in\Z$ be an
arbitrary element of $K_0$. By Lemma~\ref{cyc1},
$\chi_j(p_l)=\delta_{j,l}$ for $1\leq j,l\leq 2k+m$. Hence
$\chi_j(b)=2\pi i n_j$ for $1\leq j\leq 2k+m$. Since $b\in\A$ if and
only if $b=b^\dagger$, equalities (\ref{bla}) imply that $b\notin\A$
unless $n_j=0$ for $2k+1\leq j\leq 2k+m$ and $n_{k+j}=-n_j$ for
$1\leq j\leq k$. In the latter case $b\in\A$ since it is a linear
combination with integer coefficients of $2\pi i(p_j-p_{k+j})$ for
$1\leq j\leq k$, which all belong to $\A$ according to (\ref{cla1}).
Summarizing, we see that $K=K_0\cap \A$ is a subgroup of $(\A,+)$
generated by $k$ linearly independent vectors $2\pi i(p_j-p_{k+j})$
for $1\leq j\leq k$.

Next, according to (\ref{cla1}), for every
$\epsilon\in\{-1,1\}^{k+m}$,
$$
b_\epsilon=\sum_{j=1}^m
\epsilon_jp_{2k+j}+\sum_{j=1}^{k}\epsilon_{m+j}(p_j+p_{k+j})\in\A.
$$
Moreover, since $p_l$ are pairwise orthogonal idempotents forming a
partition of the identity, we easily see that $b_\epsilon={\bf 1}$
if $\epsilon_1={\dots}=\epsilon_{k+m}=1$ and $b_\epsilon
b_\delta=b_\alpha$ for every $\epsilon,\delta\in\{-1,1\}^{k+m}$,
where $\alpha_j=\epsilon_j\delta_j$ for $1\leq j\leq k+m$. It
immediately follows that
$G_0=\{b_\epsilon:\epsilon\in\{-1,1\}^{k+m}\}$ is a subgroup of
$\A^\star$ isomorphic to $\Z_2^{k+m}$. In particular,
$b_\epsilon^2={\bf 1}$ for each $\epsilon$. Hence $G$ is contained
in the kernel of the homomorphism $S:\A^\star\to \A^\star$,
$Sa=a^2$. First, we show that $G_0=\ker S$. Indeed, let $a\in\ker
S$. Then the square of the linear map $\widehat{a}(x)=ax$ on $\A_\C$
is the identity. It follows that $\widehat a$ is diagionalizable and
its spectrum is contained in $\{-1,1\}$. According to
Lemma~\ref{cyc1}, $a=\sum\limits_{j=1}^{2k+m}\alpha_jp_j$ with
$\alpha_j\in\{-1,1\}$. The relations $p_j=p_j^\dagger$ for $2k+1\leq
j\leq 2k+m$ and $p_j^\dagger=p_{k+j}$ for $1\leq j\leq k$ and linear
independence of $p_j$ imply now that $a$ belongs to $\A$ if and only
if $\alpha_j=\alpha_{k+j}$ for $1\leq j\leq k$. The latter means
that $a$ coincides with one of $b_\epsilon$. Thus $G_0=\ker S$. Now
let
$$
M=\{a\in\A:\chi_j(a)>0\ \ \text{for}\ \ 2k+1\leq j\leq 2k+m\}
$$
and
$$
G=\{b_\epsilon:\epsilon_{m+1}={\dots}=\epsilon_{m+k}=1\}.
$$
It is straightforward to see that $M$ is a subgroup of $\A^\star$,
that $M$ is a closed and open subset of $\A^\star$ and that $G$ is a
subgroup of $G_0$ isomorphic to $\Z_2^m$. Next observe that
$\A^\star$ is the algebraic and topological internal direct product
of its subgroups $G$ and $M$. Indeed, the equality $M\cap G=\{\bf
1\}$ is obvious. For $a\in\A^\star$, we put $\epsilon_j=1$ for
$m+1\leq j\leq m+k$, $\epsilon_j=1$ if $1\leq j\leq m$ and
$\chi_{2k+j}(a)>0$ and $\epsilon_j=-1$ if $1\leq j\leq m$ and
$\chi_{2k+j}(a)<0$. Then $b_\epsilon\in G$ and
$\chi_j(ab_\epsilon)=\chi_j(a)\chi_j(b_\epsilon)>0$ for $2k+1\leq
j\leq 2k+m$ and therefore $ab_\epsilon\in M$. It follows that $a$ is
the product of $ab_\epsilon\in M$ and $b_\epsilon\in G$. Thus
$\A^\star$ is the algebraic internal direct product of $G$ and $M$.
The fact that this product is also topological with $G$ carrying the
discrete topology easily follows from finiteness of $G$ and the fact
that $M$ is closed and open in $\A^\star$. In order to complete the
proof it remains to verify that $M=\expa(\A)$. By Lemma~\ref{expsq},
it suffices to show that $M=S(\A^\star)$. If $a\in S(\A^\star)$,
then $a=b^2$ with $b\in\A^\star$. Invertibility of $b$ implies that
$\chi_j(b)\in\R\setminus\{0\}$ for $2k+1\leq j\leq 2k+m$. Hence
$\chi_j(a)=\chi_j(b^2)=\chi_j(b)^2>0$ for $2k+1\leq j\leq 2k+m$.
Thus $a\in M$ and therefore $S(\A^\star)\subseteq M$. Now let
$\C_-=\C\setminus(-\infty,0]$ and
$$
M_0=\{a\in M:\chi_j(a)\in\C_-\ \ \text{for}\ \ 1\leq j\leq k\}.
$$
Clearly $M_0$ is an open subset of $M$. Moreover, $M_0$ is dense in
$M$. Indeed, it is easy to see that the map
$\Phi=(\chi_1,\dots,\chi_k):\A\to\C^k$ is surjective (otherwise the
characters $\chi_1,\dots,\chi_{2k}$ on $\A_\C$ are not linearly
independent). Hence the $\R$-linear map $\Phi:\A\to\C^k$ is open and
therefore $M_0=M\cap \Phi^{-1}(\C_-^k)$ is dense in $M$ since
$\C^k\setminus \C_-^k$ is nowhere dense in $\C^k$. Consider the
holomorphic branch $\psi:\C_-\to\C$ of the $\sqrt{z}$ function
satisfying $\psi(1)=1$ and let $a\in M_0$. Then
$\sigma(a)\subset\C_-$. By the usual holomorphic functional calculus
argument, we can take $b=\psi(a)\in\A_\C$. Then $b^2=a$ and since
$\psi$ is real on $(0,\infty)$, we have
$b=\psi(a)=\psi(a^\dagger)=b^\dagger$ and therefore $b\in\A$. Since
$a$ is invertible, $b\in\A^\star$ and $a=S(b)$. Thus $M_0\subseteq
S(\A^\star)$. Since the homomorphism $S:\A^\star\to M$ is a local
homeomorphism and $\ker S=G_0$ is finite, the map $S:\A^\star\to M$
is open and closed. In particular, $S(\A^\star)$ is a closed subset
of $M$ containing the dense subset $M_0$. Thus
$S(\A^\star)=\expa(\A)=M$.
\end{proof}

\section{Proof of Theorems~\ref{cyc2} and~\ref{cyc22}}

It is straightforward to see that if $G$ is a Hausdorff topological
group and $N$ is its discrete (=each point is isolated in $G$)
normal subgroup of $G$, then a subgroup $K$ of $G$ containing $N$ is
dense (respectively, nowhere dense) in $G$ if and only if $K/N$ is
dense (respectively, nowhere dense) in $G/N$. Since a homomorphism
$\phi:G\to G_0$ between Hausdorff topological groups $G$ and $G_0$
is a local homeomorphism if and only if $\ker\phi$ is discrete in
$G$ and the homomorphism $\widehat\phi:G/\ker\phi\to G_0$,
$\widehat\phi(g\ker\phi)=\phi(g)$ is a homeomorphism, we arrive to
the following fact.

\begin{lemma}\label{GH} Let $G$ and $H$ be topological groups,
$\phi:G\to H$ be a surjective homomorphism, which is also a local
homeomorphism and $K$ be a subgroup of $G$ containing $\ker\phi$.
Then $K$ is dense in $G$ if and only if $\phi(K)$ is dense in $H$.
Furthermore, $K$ is nowhere dense in $G$ if and only if $\phi(K)$ is
nowhere dense in $H$.
\end{lemma}

\begin{proof}[Proof of Theorem~$\ref{cyc2}$]
By Lemma~\ref{expc}, the kernel $K$ of the homomorphism $\expa:\A\to
\A^\star$ is a subgroup of $(\A,+)$ generated by $k=\kappa(\A)$
linearly independent elements $B_1,\dots,B_k$. By Lemma~\ref{cyc},
$\dimr \A=2n$. Pick $B_{k+1},\dots,B_{2n}\in \A$ in such a way that
$B_1,\dots,B_{2n}$ is a basis in $\A$ over $\R$. By
Corollary~\ref{subs2}, there is $B_0\in\A$ such that the additive
semigroup $H$ generated by $B_0,\dots,B_{2n}$ is dense in $\A$.
Since $\expa:\A\to\A$ is continuous and has dense in $\A$ range
$\A^\star$, $\expa(H)$ is dense in $\A$. Since $e^B_j=I$ for $1\leq
j\leq k$, we see that $\expa(H)=\expa(G)$, where $G$ is the
subsemigroup of $(\A,+)$ generated by
$\{B_0,B_{k+1},B_{k+2},\dots,B_{2n}\}$. Hence $\expa(G)$ is dense in
$\A$. Now let $x$ be a cyclic vector for $\A$. Then $\A
x=\{Ax:A\in\A\}=\C^n$. It follows that $\expa(G)x$ is dense in
$\C^n$. That is, $x$ is a hypercyclic vector for the $m$-tuple
$\{e^{B_0},e^{B_{k+1}},\dots,e^{B_{2n}}\}$ with $m=2n-k+1$ of
operators from $\A$. Thus $\A$ contains a hypercyclic $m$-tuple.

Now let $r$ be the minimal positive integer such that $\A$ contains
a hypercyclic $r$-tuple. It remains to show that $r\geq m$. Let
$\T=\{T_1,\dots,T_r\}\subset\A$ be a hypercyclic tuple and
$x\in\C^n$ be its hypercyclic vector. First, observe that each $T_j$
must be invertible. Indeed, if $1\leq j\leq r$ and $T_j$ is
non-invertible, then $L=T_j(\C^n)\neq\C^n$. Then for each
$q\in\Z_+^r$ with $q_j\neq 0$, $T_1^{q_1}\dots T_r^{q_r}x\in L$.
Since $L$, being a proper linear subspace of $\C^n$, is nowhere
dense and $x$ is a hypercyclic vector for $\T$, $\{T_1^{q_1}\dots
T_r^{q_r}x:q\in\Z_+^r,\ q_j=0\}$ is dense in $\C^n$. Hence
$\T\setminus\{T_j\}$ is a hypercyclic $(r-1)$-tuple, which
contradicts the minimality of $r$. Thus each $T_j$ is invertible.
Since $\expa:\A\to\A^\star$ is onto, we can find
$A_1,\dots,A_r\in\A$ such that $e^{A_j}=T_j$ for $1\leq j\leq r$.
Now let $M$ be the subgroup of $(\A,+)$ generated by
$A_1,\dots,A_r,B_1,\dots,B_k$. Clearly, $\expa(M)x$ contains the
$\T$-orbit of $x$ and therefore is dense in $\C^n$. Since $A\mapsto
Ax$ is an onto linear map from $\A$ to $\C^n$ and $\dim \A=n$, it is
a linear isomorphism. Hence the density of $\expa(M)x$ in $\C^n$
implies the density of $\expa(M)$ in $\A$ and therefore in
$\A^\star$. Since the additive subgroup $M$ of $\A$ contains the
kernel $K$ of $\expa$ and $\expa:\A\to\A^\star$ is a surjective
homomorphism and a local homeomorphism, Lemma~\ref{GH} provides the
density of $M$ in $\A$. Since $\dimr\A=2n$ and $M$ has $r+k$
generators, Lemma~\ref{subs1} implies that $r+k\geq 2n+1$ and
therefore $r\geq 2n-k+1=m$.
\end{proof}

\begin{proof}[Proof of Theorem~$\ref{cyc22}$]
By Lemma~\ref{expr}, the topological group $\A^\star$ is the
internal algebraic and topological direct product of its subgroups
$\expa(\A)$ and $G$, where $G$ is isomorphic to $\Z_2^m$ with
$m=\kappa_1(\A)$. Moreover, the kernel $K$ of the homomorphism
$\expa:\A\to \A^\star$ is a subgroup of $(\A,+)$ generated by
$k=\kappa_0(\A)$ linearly independent elements $B_1,\dots,B_k$. By
Lemma~\ref{cyc}, $\dimr \A=n$. Pick $B_{k+1},\dots,B_{n}\in \A$ in
such a way that $B_1,\dots,B_{n}$ is a basis in $\A$. According to
(\ref{kappa0}), $2k+m\leq n$ and therefore $m\leq n-2k\leq n-k$.
Since $\Z_2^m$  has an $m$-element generating set and $G$ is
isomorphic to $\Z_2^m$, we can pick $C_{k+1},\dots,C_n\in G$ such
that $\{C_{k+1},\dots,C_n\}$ is a generating subset of $G$. We also
set $C_j=I$ for $0\leq j\leq k$. By Corollary~\ref{subs2}, there is
$B_0\in\A$ such that the subsemigroup $N$ of $G\times\A$ generated
by $(C_0,B_0),\dots,(C_{2n},B_{2n})$ is dense in $G\times \A$. Since
$\A^\star$ is the internal algebraic and topological direct product
of its subgroups $\expa(\A)$ and $G$ and $\expa$ is a local
homeomorphism, the homomorphism $\Phi:G\times \A\to\A^\star$,
$\Phi(C,A)=Ce^A$ is surjective and is a local homeomorphism.
Moreover, $\ker\Phi=\{I\}\times K$. Then $\Phi(N)$ is dense in
$\A^\star$ and therefore is dense in $\A$. Since $C_j=I$ and
$e^{B_j}=I$ for $1\leq j\leq k$, $\Phi(N)$ is precisely $\F_\T$ for
$\T=(C_0e^{B_0},C_{k+1}e^{B_{k+1}},\dots,C_{n}e^{B_n})$. The density
of $\F_\T$ in $\A$ implies that each cyclic vector $x$ for $\A$ is
also a hypercyclic vector for $\T$, which happens to be an $r$-tuple
of elements of $\A$ with $r=n-k+1$.

Now let $p$ be the minimal positive integer such that $\A$ contains
a $p$-tuple with an orbit, which is not nowhere dense. It remains to
show that $p\geq r$. Let $\T=\{T_1,\dots,T_p\}\subset\A$ for which
there is $x\in\R^n$, whose orbit $O(\T,x)$ is not nowhere dense.
Exactly as in the proof of Theorem~\ref{cyc22}, minimality of $p$
implies that each $T_j$ is invertible. Let
$\T^{[2]}=\{T_1^2,\dots,T_p^2\}$. Then $O(\T,x)$ is the union of
$O(\T^{[2]},x_\epsilon)$, where $\epsilon\in\{0,1\}^p$ and
$x_\epsilon=T_1^{\epsilon_1}\dots T_p^{\epsilon_p}x$. Thus there is
$y\in\{x_\epsilon:\epsilon\in\{0,1\}^p\}$ such that $O(\T^{[2]},y)$
is not nowhere dense. Then $y$ is a cyclic vector for $\A$ and
therefore the map $T\mapsto Ty$ from $\A$ to $\R^n$ is onto. Since
$\dim\A=n$, this map is a linear isomorphism. Hence the semigroup
$\F_{\T^{[2]}}$ is not nowhere dense in $\A$. Thus the subgroup
$\F^0_{\T^{[2]}}$ of $\A^\star$ generated by $\T^{[2]}$ is also not
nowhere dense in $\A$. By Lemma~\ref{expsq}, there are
$A_1,\dots,A_p\in\A$ such that $T_j^2=e^{A_j}$ for $1\leq j\leq p$.
Let $Q$ be the additive subgroup of $\A$ generated by
$A_1,\dots,A_p,B_1,\dots,B_k$. Since $e^{A_j}=T_j^2$ for $1\leq
j\leq p$ and $e^{B_j}=I$ for $1\leq j\leq k$,
$\expa(Q)=\F^0_{\T^{[2]}}$ is a subgroup of $\A^\star$, which is not
nowhere dense. Since $\expa:\A\to \expa(\A)$ is a surjective
homomorphism and a local homeomorphism and $Q$ contains $\ker\expa$,
Lemma~\ref{GH} implies that $Q$ is not nowhere dense in $\A$. Since
$Q$ is generated by $p+k$ elements and $\dim\A=n$, Lemma~\ref{subs1}
implies that $p+k\geq n+1$ and therefore $p\geq n-k+1=r$.
\end{proof}

\section{Proof of Propositions~\ref{cy} and~\ref{F4}}

\begin{proof}[Proof of Proposition~$\ref{cy}$] Consider the $3$-dimensional
linear subspace $\A$ of $L(\K^3)$ defined as follows:
$$
\A=\{A_z:z\in\K^3\},\ \ \text{where}\ \
A_z=\begin{pmatrix}z_1&0&0\\
z_2&z_1&0\\
z_3&0&z_{1}
\end{pmatrix}
$$
It is easy to see that $\A$ is a commutative subalgebra of $L(\K^3)$
and $\A$ is cyclic with the cyclic vector $(1,0,0)$. Moreover, there
is exactly one character on $\A$: $A_z\mapsto z_1$. Hence
$\kappa(\A)=1$ if $\K=\C$ and $\kappa_0(\A)=0$, $\kappa_1(\A)=1$ if
$\K=\R$. According to Theorems~\ref{cyc2} and~\ref{cyc22}, $\A$
contains a hypercyclic $6$-tuple (respectively, $4$-tuple) $\T$ if
$\K=\C$ (respectively, $\K=\R$). Since any hypercyclic tuple of
operators on $\K^3$ is contained in precisely one commutative cyclic
subalgebra of $L(\K^3)$, which is the unital subalgebra genrated by
the tuple, $\A_\T=\A$. By Lemma~\ref{cyc}, any operator commuting
with $\T$ belongs to $\A$. Thus it suffices to show that $\A$
contains no cyclic operators. Indeed, assume that there is
$z\in\K^3$ such that $A_z$ is cyclic. Cyclicity of $A_z$ implies
cyclicity of $A_z-z_{1}I$. Since the last two columns of
$A_z-z_{1}I$ are zero, the rank of  $A_z-z_{1}I$ is at most $1$.
Since the rank of every cyclic operator on $\C^n$ is at least $n-1$,
we have arrived to a contradiction.
\end{proof}

\begin{proof}[Proof of Proposition~$\ref{F4}$]
Let $a_1,a_2,a_3>0$, $b_1,b_2,b_3\in\R$ and
$T_j=\begin{pmatrix}a_j&b_j\\ 0&a_j\end{pmatrix}\in L(\R^2)$ for
$1\leq j\leq 3$. It is straightforward to verify that for every
$x=(x_1,x_2)\in\R^2$ and each $n_1,n_2,n_3\in\Z_+$,
$$
T_1^{n_1}T_2^{n_2}T_3^{n_3}x=
a_1^{n_1}a_2^{n_2}a_3^{n_3}\bigl(x_1+x_2\bigl({\textstyle
\frac{b_1n_1}{a_1}+\frac{b_2n_2}{a_2}+\frac{b_3n_3}{a_3}}\bigr),x_2\bigr).
$$
Thus for $\T=(T_1,T_2,T_3)$, the $\T$-orbit $O(\T,x)$ of $x$ is
contained in the half-plane $\Pi=\{(s,t)\in\R^2:t>0\}$ provided
$x_2>0$. It remains to show that $a_j>0$ and
$b_j\in\R\setminus\{0\}$ can be chosen in such a way that $O(\T,x)$
is dense in $\Pi$ if $x_2>0$.

According to the above display, the density of $O$ in $\Pi$ is
equivalent to the density of
$$
N=\bigl\{\bigl({\textstyle
\frac{b_1n_1}{a_1}+\frac{b_2n_2}{a_2}+\frac{b_3n_3}{a_3}},
 a_1^{n_1}a_2^{n_2}a_3^{n_3}\bigr):n_1,n_2,n_3\in\Z_+\bigr\}
$$
in $\Pi$. Since the map $(t,s)\mapsto (t,\ln s)$ is a homeomorphism
from $\Pi$ onto $\R^2$, the density of $O$ in $\Pi$ is equivalent to
the density of
$$
N_0=\bigl\{{\textstyle n_1(\frac{b_1}{a_1},\ln a_1)+
n_2(\frac{b_2}{a_2},\ln a_2)+n_3(\frac{b_3}{a_3},\ln
a_3)}:n_1,n_2,n_3\in\Z_+\bigr\}
$$
in $\R^2$. Now choose any $a_1,a_2>0$ and
$b_1,b_2\in\R\setminus\{0\}$ such that the vectors
$(\frac{b_1}{a_1},\ln a_1)$ and $(\frac{b_2}{a_2},\ln a_2)$ are
linearly independent. By Corollary~\ref{subs2}, there are $a_3>0$
and $b_3\in\R\setminus\{0\}$ such that the additive subsemigroup
$N_0$ in $\R^2$ generated by $(\frac{b_1}{a_1},\ln a_1)$,
$(\frac{b_2}{a_2},\ln a_2)$ and $(\frac{b_3}{a_3},\ln a_3)$ is dense
in $\R^2$. For this choice of $a_j$ and $b_j$, the orbit $O(\T,x)$
is dense in $\Pi$ whenever $x_2>0$.
\end{proof}

\section{Remarks}

It is easy to extend Proposition~\ref{cy} to provide hypercyclic
tuples $\T$ on $\K^n$ with $n\geq 3$ such that there are no cyclic
operators commuting with $\T$. One can also produce an infinite
dimensional version. Namely, let $X$ be any infinite dimensional
separable Fr\'echet space over $\K$ and $Y$ be a closed linear
subspace of $X$ of codimension 3. Then $X$ can be naturally
interpreted as $Y\oplus \K^3$. Take the tuple $\T$ of operators on
$\K^3$ constructed in the proof of Proposition~\ref{cy}. Since every
separable infinite dimensional separable Fr\'echet space possesses a
mixing operator \cite{bama-book}, we can pick a mixing $S\in L(Y)$.
It is now easy to verify that the tuple $\T_0=\{S\oplus T:T\in\T\}$
is a hypercyclic tuple of operators on $Y\oplus \K^3=X$ consisting
of the same number of operators as $\T$. Take any $T\in\T$. Then
there is a unique $\lambda\in\K$ such that $(T-\lambda
I_{\K^3})^3=0$. Since $(S-\lambda I_Y)^3(Y)$ is dense in $Y$, we
have $Y=\overline{(T_0-\lambda I_X)^3(X)}$, where $T_0=S\oplus T$.
Since $T_0\in\T_0$, the equality $Y=\overline{(T_0-\lambda
I_X)^3(X)}$ implies that $Y$ is invariant for every operator in
$L(X)$ commuting with $\T_0$. By factoring $Y$ out, we see that the
existence of a cyclic operator on $X$ commuting with $\T_0$ implies
the existence of a cyclic operator on $\K^3$ commuting with $\T$.
Since the latter does not exist, we arrive to the following result.

\begin{proposition}\label{frec} Every infinite dimensional separable
complex $($respectively, real$\,)$ Fr\'echet space admits a
hypercyclic $6$-tuple $($respectively, a $4$-tuple$)$ $\T$ of
operators such that there are no cyclic operators commuting with
$\T$.
\end{proposition}

Summarizing, we notice that any separable Fr\'echet space $X$ of
dimension $\geq 3$ possesses a hypercyclic tuple $\T$ of operators
such that there are no cyclic operators commuting with $\T$. Since
any operator on a one-dimensional space is cyclic, this does not
carry on to the case $\dim X=1$. Neither it does to the case $\dim
X=2$. Indeed, the only non-cyclic operators on $\K^2$ have the form
$\lambda I$ with $\lambda\in\K$. Since a hypercyclic tuple of
operators on $\K^2$ can not consist only of such operators, we
arrive to the following proposition.

\begin{proposition}\label{2dim} Every hypercyclic tuple
of operators on $\K^2$ contains a cyclic operator.
\end{proposition}

Theorem~5.5 in \cite{feld} provides a sufficient condition on a
commuting tuple $\T$ of continuous linear operators on a locally
convex topological vector space $X$ to have each orbit either dense
or nowhere dense. It is worth noting that the local convexity
condition can be dropped in exactly the same way as it is done by
Wengenroth \cite{ww}. One just has to replace the point spectrum
$\sigma_p(A^*)$ of the dual operator $A^*$ of each $A\in L(X)$ in
the statement by the set $\{\lambda\in\K:\overline{(A-\lambda
I)(X)}\neq X\}$, which coincides with $\sigma_p(A^*)$ in the locally
convex case.

Finally, we note that Feldman \cite{feld} proved that there are no
hypercyclic tuples of normal operators on a separable infinite
dimensional complex Hilbert space ${\cal H}$. On the other hand,
Bayart and Matheron \cite{bm} constructed a unitary operator $U$ on
${\cal H}$ and $x\in{\cal H}$ such that $\{zU^nx:z\in\C,\
n\in\Z_+\}$ is dense in ${\cal H}$ endowed with its weak topology
$\sigma$. According to Feldman \cite{feld}, there are $a,b\in\C$
such that $\{a^nb^m:n,m\in\Z_+\}$ is dense in $\C$. It follows that
$(aI,bI,U)$ is a commuting triple of normal operators on $\cal H$,
which is a hypercyclic triple of operators on $({\cal H},\sigma)$.


\small\rm

\vskip1truecm

\scshape

\noindent Stanislav Shkarin

\noindent Queens's University Belfast

\noindent Pure Mathematics Research Centre

\noindent University road, Belfast, BT7 1NN, UK

\noindent E-mail address: \qquad {\tt s.shkarin@qub.ac.uk}

\end{document}